\documentclass[a4paper,10pt,reqno]{amsart}

\usepackage{amssymb}

\usepackage[utf8]{inputenc}
\usepackage[T1]{fontenc}

\usepackage{todonotes}
\usepackage{undertilde,tikz,enumerate}
\usepackage{upgreek}
\DeclareFontFamily{U}{mathx}{\hyphenchar\font45}
\DeclareFontShape{U}{mathx}{m}{n}{
      <5> <6> <7> <8> <9> <10>
      <10.95> <12> <14.4> <17.28> <20.74> <24.88>
      mathx10
      }{}
\DeclareSymbolFont{mathx}{U}{mathx}{m}{n}
\DeclareFontSubstitution{U}{mathx}{m}{n}
\DeclareMathAccent{\widecheck}{0}{mathx}{"71}
\DeclareMathAccent{\wideparen}{0}{mathx}{"75}

\subjclass[2010]{}

\newcommand{\N}{\mathbb{N}}

\renewcommand{\restriction}{{{\upharpoonright}}}

\newcommand{\struc}[1]{\langle #1 \rangle}
\newcommand{\alg}[1]{\mathrm{\mathbf #1}} 
\newcommand{\str}[1]{#1}

\newcommand{\var}[1]{\mathcal{#1}}
\newcommand{\classop}[1]{\mathbb{#1}}

\newcommand{\cate}[1]{\mathcal{#1}}
\newcommand{\card}[1]{\vert #1 \vert}

\newcommand{\dom}{\mathrm{dom}}
\newcommand{\im}{\mathrm{Im}}
\newcommand{\Con}[1]{\mathrm{Con}(#1)}
\newcommand{\pro}{\mathrm{\mathbf{pro}}}

\renewcommand{\mod}{\, \mathrm{mod}\, }

\newcommand{\me}{\mathrm{\bf m}}

\newcommand{\lowerset}[1]{#1{\downarrow}}
\newcommand{\uperset}[1]{#1{\uparrow}}
\renewcommand{\emptyset}{\varnothing}

\newtheorem{thm}{Theorem}[section]
\newtheorem{cor}[thm]{Corollary}
\newtheorem{lem}[thm]{Lemma}

\newtheorem{prop}[thm]{Proposition}
\newtheorem{ass}[thm]{Assumption}
\theoremstyle{definition}
\newtheorem{defn}[thm]{Definition}
\newtheorem{nota}[thm]{Notation}
\theoremstyle{remark}
\newtheorem{rem}[thm]{Remark}
\newtheorem{ex}[thm]{Example}

\begin{document}
\title[Ext. maps to profinite compl. in finit. gen. quasivarieties]{Extending maps to profinite completions in finitely generated quasivarieties}

\author{Georges \textsc{Hansoul}}
\address{Département de Mathématiques, Université de Liège, Grande Traverse, 12, 4000 Liège, Belgium}
\email{g.hansoul@uliege.be}
\author{Bruno \textsc{Teheux}}
\address{Mathematics Research Unit, FSTC, University of Luxembourg, 6, Rue Coudenhove-Kalergi, L-1359 Luxembourg, Luxembourg}
\email{bruno.teheux@uni.lu}

\begin{abstract}
We consider the problem of extending maps from algebras to their profinite completions in finitely generated quasivarieties. Our developments are based on the construction of the profinite completion of an algebra as its natural extension. We provide an extension which is a multi-map and we study its continuity properties, and  the conditions under which it is a map.
\end{abstract}
\maketitle

\section{Introduction}

This paper is a contribution to the study of profinite completions in internally residually finite prevarieties. A class $\var{A}$ of algebras is called  \cite{Davey2011} an \emph{internally residually finite prevariety} (\emph{IRF-prevariety} for short) if there is a  set $\mathcal{M}$ of finite algebras such that $\var{A}=\classop{ISP}(\mathcal{M})$. Every algebra $\alg{A}$ of an IRF-prevariety $\var{A}$ embeds in its \emph{$\var{A}$-profinite completion $\pro_\var{A}(\alg{A})$}, which is defined  as the inverse limit of the inverse system of the finite quotients of $\alg{A}$ that belongs to $\var{A}$, with natural homomorphisms as bonding maps (see Section \ref{sec:prel} for details). In what follows, we limit ourself to those IRF-prevareties $\var{A}$ that are \emph{finitely generated quasivarieties}, \emph{i.e.}, for which there is a \emph{finite set} $\mathcal{M}$ of finite algebras such that $\var{A}=\classop{ISP}(\mathcal{M})$. Moreover, we assume that $\var{A}=\classop{ISP}(\{\alg{M}\})$, but this restriction is a matter of convenience: we claim that our developments admit the obvious generalization to the multi-sorted case where $\mathcal{M}=\{\alg{M}_1, \ldots, \alg{M}_n\}$. 

It is proved in \cite{Davey2011} that $\pro_\var{A}(\alg{A})$ is isomorphic to the \emph{natural extension $\alg{A}^\delta$ of $\alg{A}$}, that is,  the topological closure of $e_\alg{A}(\alg{A})$ in $\alg{M}_\iota^{\alg{A}^*}$, where $\alg{A}^*=\var{A}(\alg{A}, \alg{M})$, where $e_\alg{A}\colon \alg{A} \to \alg{M}^{\alg{A}^*}$ is the evaluation map defined as $e_{\alg{A}}(a)(\phi)=\phi(a)$, and where $\iota$ is the discrete topology on $\alg{M}$ (this representation result actually holds in any IRF-prevariety). Moreover, if $\utilde{M}$ is a discrete structure that yields a natural duality for $\var{A}$, and if $\alg{A}^*$ is considered as a (closed) substructure of $\utilde{M}^\alg{A}$, then $\alg{A}^\delta$ can  be concretely computed as the algebra of structure preserving maps from $\alg{A}^*$ to $\utilde{M}$ \cite[Theorem 4.3]{Davey2011}. With these results in mind, we adopt the notation $\alg{A}^\delta$ to denote $\pro_\var{A}(\alg{A})$ for the remaining of the paper.

We consider the following problem: given $\alg{A}, \alg{B} \in \var{A}$ and a map $u \colon \alg{A} \to \alg{B}$, how to define a `reasonable' extension $u^\delta \colon \alg{A}^\delta \to \alg{B}^\delta$ of $u$? Such an extension would allow to study profinite completions of expansions  of $\var{A}$-algebras, and preservation of equations through profinite completions. This problem has a well known solution  \cite{Gehrke2004} in the particular case where $\var{A}=\var{DL}=\classop{ISP}(\alg{2})$ is the  variety of bounded distributive lattices,  in which profinite completions coincide with canonical extensions. In this particular case, the theory of canonical extensions provides with a lower and an upper extension of any map $u\colon \alg{L} \to \alg{L'}$ to the canonical extensions of $\alg{L}$ and $\alg{L}'$. However, in the more general setting of non lattice-based algebras, no method of extension of maps from algebras to their profinite completions has  yet been developed.

The paper is organized as follows. In section \ref{sec:prel}, we recall some results about profinite completions in IRF-prevarieties, and we set up the notations. In section \ref{sec:topo}, we introduce a new topology $\delta$ on $\alg{A}^\delta$ such that $\alg{A}$ is the largest discrete subspace of $\alg{A}^\delta_\delta$. We prove that this topology boils down to the existing one  \cite{Geh00, Gehrke2011} in the specific case $\var{A}=\var{DL}$. Finally, we prove that if $\utilde{M}$ yields a logarithmic full duality for $\var{A}$, then the construction of profinite completions (\emph{alias}, natural extensions) commutes with the one of finite Cartesian products. We generalize this result to Boolean products in the Appendix.

Section \ref{sec:maps} is the core of the paper. We work under the more restrictive assumption that there is a discrete topological structure $\utilde{M}$ that yields a  logarithmic duality for $\alg{A}$ and that $\utilde{M}$ is injective in the dual category $\classop{IS}_c\classop{P}(\utilde{M})$. Given a map $u\colon \alg{A} \to \alg{B}$, we use the topology $\delta$ to define an extension $\widetilde{u}$ of $u$ on $\alg{A}^\delta$. In  general, the map $\widetilde{u}$ is not valued on $\alg{B}^\delta$ but $\widetilde{u}(x)$ is a closed subspace of $\alg{B}^\delta_\iota$ for every $x\in \alg{A}^\delta$ (where $\iota$ is the topology inherited from $\utilde{M}^{\var{A}(\alg{B}, \alg{M})}$). It means that $\widetilde{u}$ has to be considered as a multi-map between $\alg{A}^\delta$ and $\alg{B}^\delta$ rather than a map. Nevertheless,  under additional continuity assumptions, we show that $\widetilde{u}$ can be turned into a map valued in $\alg{B}^\delta$, a property that we call \emph{smoothness}. 

In Section \ref{sec:compo}, we study how the construction of $\widetilde{u}$ interacts with function compositions. We illustrate our developments by considering a sample case, namely, the case where $\var{A}$ is the variety of median algebras  (a non lattice-based variety). In particular, we  exhibit an example of smooth map that is not a homomorphism nor an operation of the type of the algebra. We also prove that  median algebras whose profinite completion is Boolean are exactly the Boolean powers of the 2-element  median algebra.

In section \ref{sec:order}, we consider the special case where $\alg{M}$ can be equipped with a total order in such a way that every $\alg{A}\in \var{A}$ can be considered locally has a (semi)lattice. We also show how the construction of $\widetilde{u}$ shed lights on the existence of an upper and a lower extension of $u$ in the case of distributive lattices. We close the paper by some concluding remarks and topics of further research.

\section{Preliminaries}\label{sec:prel}

We work under the general setting of \cite{NDFWA}. Let $\mathcal{M}=\{\alg{M}_1, \ldots, \alg{M}_n\}$ be a finite set of finite algebras of the same type,  and let $\var{A}$ be the prevariety $\classop{ISP}(\mathcal{M})$. In what follows we assume for convenience that $\mathcal{M}=\{\alg{M}\}$.  We denote by $\utilde{M}$ an alter-ego of $\alg{M}$, \emph{i.e.}, a topological structure
\[
\utilde{M}=\struc{M, G\cup H \cup R, \iota},
\]
where $\iota$ is the discrete topology on $M$, and $G$, $H$ and $R$ are respectively a set (possibly empty) of algebraic operations, algebraic partial operations (with nonempty domain), and algebraic (nonempty) relations on $\alg{M}$, respectively. We use $\cate{X}$ to denote the topological prevariety $\classop{IS}_c\classop{P}(\utilde{M})$, that is, the class of topological structures that are isomorphic to a closed substructure of a nonempty power of $\utilde{M}$. For any $X, Y \in \cate{X}$ we denote by $\cate{X}(X,Y)$ the set of the structure preserving continuous maps $f:X\to Y$. We use $X_*$ to denote  $\cate{X}(X,\utilde{M})$.

For any $\alg{A}\in \var{A}$, we denote by $\alg{A}^*$ the set $\var{A}(\alg{A}, \alg{M})$ of the homomorphisms from $\alg{A}$ to $\alg{M}$.  The Preduality Theorem (\cite[Theorem 5.2]{NDFWA}) states that  if $\alg{A}^*$ inherits the structure and the topology from $\utilde{M}^A$, then $\alg{A}^*\in \cate{X}$. Moreover, the evaluation map 
\[
e_\alg{A}:\alg{A}\to (\alg{A}^*)_*:a \mapsto e_{\alg{A}}(a):\phi \mapsto \phi(a)
\]
is an $\var{A}$-embedding.  Similarly, for any $X\in \cate{X}$ the map
\[
\epsilon_{X}:X\to (X_*)^*: \phi \mapsto \epsilon_{X}(\phi):x \mapsto x(\phi)
\]
is a $\cate{X}$-embedding.

If $\tau$ is a topology on a set $X$, we denote by $\struc{X,\tau}$ or $X_\tau$ the corresponding topological space. In particular, we denote by $\alg{M}_\iota$  the topological algebra obtained by equipping $\alg{M}$ with the discrete topology. For any set $X$, the notation $\alg{M}^X_\iota$ stands for the power algebra $\alg{M}^X$ equipped with the product topology induced by $\iota$ on $M$.  We denote by $\var{A}_\iota$ the category of topological algebras that are isomorphic to a closed subalgebra of a nonempty power of $\alg{M}_\iota$ with continuous homomorphisms as arrows. For every $\alg{A}\in \var{A}$, the map $e_\alg{A}$ identifies $\alg{A}$ with the subspace $e_\alg{A}(\alg{A})$  of $\alg{M}_\iota^{\alg{A}^*}$, and we usually consider $\alg{A}$ up to this identification.
   
\begin{defn}[\cite{Davey2011}]\label{def:nat}
The \emph{natural extension} of $\alg{A}\in \var{A}$, in notation $\alg{A}^\delta$,  is the topological closure of $\alg{A}$ in $\alg{M}^{\alg{A}^*}_\iota$. The algebra  $\alg{A}^\delta$ is turned into an element $\alg{A}^\delta_\iota$ of $\var{A}_\iota$ by considering it as a subspace of  $\alg{M}^{\alg{A}^*}_\iota$.
\end{defn}

When $\utilde{M}$ yields a natural duality for $\var{A}$, \emph{i.e.}, when the map $e_{\alg{A}}\colon \alg{A} \to (\alg{A}^*)_*$ is an isomorphism for every $\alg{A}\in \var{A}$, then Proposition \ref{prop:init} shows how to explicitly construct $\alg{A}^\delta$ from $\alg{A}^*$ without relying on any notion of (topological) limit.
For any topological structure $\str{X}$ and any topological algebra $\alg{A}$, we denote respectively by $\str{X}^\flat$ and $\alg{A}^\flat$ the structure obtained from $\str{X}$ and the algebra obtained from $\alg{A}$ by dropping the topology. 
We denote by $\cate{X}^\flat$ the category whose objects are the $X^\flat$ where $X\in \cate{X}$ with structure preserving maps as arrows. By abuse of notation, we write $\cate{X}^\flat(\str{X}, \str{Y})$ instead of $\cate{X}^\flat(\str{X}^\flat, \str{Y}^\flat)$. Note that $\cate{X}^\flat(\alg{A}^*, \utilde{M})$ is  a closed subalgebra of $\alg{M}^{\alg{A}^*}_\iota$ for every $\alg{A}\in \var{A}$.

\begin{prop}[{\cite[Theorems 3.6 and 4.3]{Davey2011}}]\label{prop:init} Assume that $\alg{A}\in \var{A}$.
\begin{enumerate}
\item\label{it:init01} The definition of $\alg{A}^\delta_\iota$ is independent of the algebraic structure $G\cup H\cup R$ used to defined $\utilde{M}$ and of the algebra $\alg{M}$ used to define $\var{A}$.
\item\label{it:init02} If in addition $\utilde{M}$ yields a duality for $\var{A}$ then $\alg{A}^\delta_\iota$ is isomorphic to $\struc{\cate{X}^\flat(\alg{A}^*, \utilde{M}), \iota}$.
\end{enumerate}
\end{prop}

Recall that for any $\alg{A}\in \var{A}$, the family $\{A/\theta \mid  \theta \in \Con{A} \text{ and } A/\theta\in \var{A}  \text{ is finite}\}$ with the natural bonding maps $\phi_{\theta, \theta'}\colon x/\theta \mapsto x/\theta'$ for every $\theta \leq \theta'$ forms an inverse system, the inverse limit of which is called the \emph{$\var{A}$-profinite completion of $\alg{A}$} (or simply the \emph{profinite completion of $\alg{A}$}) and is denoted by $\pro_{\var{A}}(\alg{A})$. Any $\alg{A}\in \var{A}$ embeds in $\pro_{\var{A}}(\alg{A})$. If in addition $\var{A}$ is a variety, then $A/\theta \in \var{A}$ for every congruence of $\alg{A}$, and the construction of $\pro_{\var{A}}(\alg{A})$ does not rely on $\var{A}$ and is commonly denoted by $\widehat{\alg{A}}$. The following result, which follows from \cite[Theorem 3.6]{Davey2011},  states that under our assumption of a finitely generated prevariety $\var{A}$, the $\var{A}$-profinite completion of $\alg{A}\in \var{A}$ coincides with its natural extension $\alg{A}^\delta$.
\begin{prop}\label{prop:pro}
If $\alg{A} \in\var{A}$, then there is an isomorphism between $\pro_{\var{A}}(\alg{A})$ and $\alg{A}^\delta$ that fixes $\alg{A}$. 
\end{prop}
Informally speaking, Proposition \ref{prop:pro} shows that natural extension is a tool to compute profinite completions.

We close the section by introducing some notation. We write $F\Subset X$ if $F$ is a finite subset of $X$. If $\tau$ is a topology on $X$ and $x\in X$, then we denote by $\tau_x$ the set of open $\tau$-neighborhoods of $x$. If $b\in X_\tau^Z$ for some $Z$ and if  $F\Subset Z$, then we denote by $[b|F]$ the basic open set $\{y\in X^Z \mid y\restriction_F=b\restriction_F\}$ of $X_\tau^Z$.

If $(X, \leq)$ is an ordered set and $x\in X$, then we denote by $x{\uparrow}$ and $x{\downarrow}$ the up-set and the down-set generated by $x$, respectively.

\section{The topology $\delta$ for profinite completions}\label{sec:topo}

In the distributive lattice-based setting, it is well known that the  topology $\iota$ that naturally equips the canonical extension of a DL $\alg{A}$
can be enriched into a finer topology in which $\alg{A}$ is definable as the algebra of isolated points. Authors have used various notations for this topology: \textsc{Gehrke} and \textsc{Jónsson} denote it by $\sigma$  in \cite{Gehrke2004} and \textsc{Gehrke} and \textsc{Vosmaer} denote it by $\delta$ in \cite{Gehrke2011}. We aim at defining a similar topology in the more general setting of a finitely generated prevariety $\var{A}$ and $\var{A}$-profinite completions.

\subsection{A topology to define $\alg{A}$ in its profinite completion}

If $X, Y \in \cate{X}$ we denote by  $\cate{X}_p(X,Y)$ the set of partial morphisms from $X$ to $Y$, \emph{i.e.}, the set of the maps $f\colon\dom(f)\to Y$ where  $\dom(f)$ is a closed substructure of $X$ and where $f\in \cate{X}(\dom(f), Y)$.

\begin{defn}\label{defn:topo}
If $\alg{A}\in \var{A}$ and  $f \in \cate{X}_p(\alg{A}^*, \utilde{M})$, we set
\[
O_{f}=\{x \in \cate{X}^\flat(\alg{A}^*, \utilde{M})\mid x\supseteq f\}.
\]
Then, we denote by  $\Delta_{\alg{A}}$, or simply $\Delta$, the family
\[
\Delta=\{O_f\mid f \in \cate{X}_p(\alg{A}^*, \utilde{M})\}.
\]
 The topology  $\delta$ is defined as the topology generated by $\Delta$, and we denote by $\alg{A}_\delta$ the topological algebraic structure obtained by equipping $\alg{A}$ with $\delta$.
\end{defn}

\begin{rem}
\begin{enumerate}
\item If $\utilde{M}$ is injective in $\cate{X}$, then $\Delta$ is equal to the family of the sets $O_{K, a}:=O_{e_\alg{A}(a)\restriction_K}$ where $a\in \alg{A}$ and $K$ is a closed substructure of $\alg{A}^*$.
\item  It is not always possible to compare topologies $\delta$ and $\iota$. Nevertheless,  we have $\iota\subseteq\delta$ if any finite subset of $\alg{A}^*$ generates a finite  substructure in $\alg{A}^*$. In particular, we have $\iota \subseteq \delta$  if $\utilde{M}$ is a purely relational structure. 

\end{enumerate}
\end{rem}

Recall that a strong duality is said to be \emph{logarithmic} if (finite) coproducts in the dual category (they always exist since they are dual to products) are given by the direct unions, that is, disjoint unions with constants amalgamated (see section 6.3 in \cite{NDFWA}). 

\begin{lem}\label{lem:base}
If $\utilde{M}$ yields a logarithmic duality for $\var{A}$, then $\Delta$ is a basis of $\delta$.
\end{lem}
\begin{proof}
Let $\alg{A}\in \var{A}$ and  $f, g \in \cate{X}_p(\alg{A}^*,\utilde{M})$. We prove that $O_f \cap O_g\in \Delta$ or $O_f \cap O_g = \varnothing$. First, we note that $\dom(f)\cup \dom(g)$ is a substructure of $\alg{A}^*$. Indeed, if $i_h$ is the inclusion map $i_h:\dom(h)\rightarrow \alg{A}^*$ for $h\in \{f,g\}$, and if $s_h$ is the canonical embedding from $\dom(h)$ into $\dom(f)\amalg \dom(g)$ for $h \in \{f,g\}$, then there is a morphism $i\colon\dom(f)\amalg \dom(g)\rightarrow \alg{A}^*$ such that $i \circ s_f=i_f$ and $i \circ s_g=i_g$.  It follows that $\im(i)=\dom(f)\cup\dom(g)$ is a closed substructure of $\alg{A}^*$.

If $f \cup g$ is not a function, \emph{i.e.}, if $f$ and $g$ do not coincide on $\dom(f)\cap\dom(g)$, or if $f \cup g$ is a function that does not belong to $\cate{X}^\flat(\dom(f)\cup\dom(g), \utilde{M})$, then $O_f\cap O_g=\emptyset$. 

 If $f \cup g \in \cate{X}^\flat(\dom(f)\cup\dom(g), \utilde{M})$, then it also belongs to $\cate{X}(\dom(f) \cup \dom(g), \utilde{M})$ by continuity of $f$  and $g$.  It follows that $O_f\cap O_g=O_{f \cup g}$.
\end{proof}

\begin{lem}\label{lem:bestcase} Assume that $A\in \var{A}$.
\begin{enumerate}
\item The elements of $\cate{X}(\alg{A}^*, \utilde{M})$ are isolated points in $\struc{\cate{X}^\flat(\alg{A}^*, \utilde{M}), \delta}$.
\item If $\utilde{M}$ is injective in $\cate{X}$ and if $\Delta$ is a basis of $\delta$, then $(\alg{A}^*)_*$ is dense in $\struc{\cate{X}^\flat(\alg{A}^*, \utilde{M}), \delta}$.
\item\label{it:sep} If $\utilde{M}$ is injective in $\cate{X}$ and yields a duality for $\var{A}$ and if $\Delta$ is a basis of $\delta$, then $\alg{A}$ is a discrete dense subspace of $\alg{A}^\delta_\delta$.
\end{enumerate}
\end{lem}
\begin{proof}
(1) If $x \in (\alg{A}^*)_*$ then $O_x=\{x\}\in \delta$.

(2) Let $f\in \cate{X}_p(\alg{A}^*, \utilde{M})$.
Since $\utilde{M}$ is injective in $\cate{X}$, there is an $a\in (\alg{A}^*)_*$ such that $a=f$ on $\dom(f)$. It means by definition of $\delta$ that $a\in O_f \cap (\alg{A}^*)_*$.

(3) We know by (1) and (2) that $e_{\alg{A}}(\alg{A})=(\alg{A}^*)_*$ is the subspace of isolated points of $\alg{A}^\delta_\delta$. The conclusion follows from Proposition \ref{prop:init} (\ref{it:init02}).
\end{proof}

By combining Lemmas \ref{lem:base} and \ref{lem:bestcase}, we obtain the following proposition.

\begin{prop}\label{prop:bestcase}
If $\utilde{M}$ is injective in $\cate{X}$ and yields a logarithmic duality for $\var{A}$, then $\alg{A}$ is the subspace of isolated points of $\alg{A}^\delta_\delta$ for every $\alg{A}\in \var{A}$.
\end{prop}

Let $\var{DL}$ be the variety of bounded distributive lattices, that is, $\var{DL}=\classop{ISP}(\alg{2})$, where $\alg{2}=\struc{\{0,1\}, \vee, \wedge}$ is the two-element lattice. Recall that  $\utilde{\alg{2}}:=\struc{\{0,1\}, \leq, \iota}$ where $0\leq 1$ yields a logarithmic natural strong duality for $\var{DL}$,  known as \textsc{Priestley} duality. If $\alg{L}\in \var{DL}$, then $\alg{L}^\delta$ coincides with the canonical extension of $\alg{L}$, which can be constructed by Proposition \ref{prop:init}  as the lattice of decreasing subsets of the  $\alg{L}^*$. In \cite{Gehrke2004}, the authors introduce a topology $\delta'$ on $\alg{L}^\delta$ (denoted by $\sigma$ in  \cite{Gehrke2004} and by $\delta$ in \cite{Gehrke2011}), and use this topology to extend maps between distributive lattices to their canononical extensions (\emph{i.e.}, their profinite completions). Recall that a basis of $\delta'$ is given by the sets $[F,O]$ where $F$ is  a closed element of $\alg{L}^\delta$ (\emph{i.e.}, a closed decreasing subset of $\alg{L}^*$), and  $O$ is an open element of $\alg{L}^\delta$ (\emph{i.e.}, an open decreasing subset of  $\alg{L}^*$). 
In the next proposition, we prove that the topology $\delta$ defined in  Definition \ref{defn:topo} coincides with  $\delta'$.

\begin{prop}\label{prop:boils} 
If $\alg{L}\in \var{DL}$, then $\delta(\alg{L}^\delta)=\delta'(\alg{L}^\delta)$.
\end{prop}
\begin{proof}
First, we prove that $\delta'\subseteq \delta$. Let $F$ and $O$ be  a closed and an open element of $\alg{L}^\delta$, respectively, and assume  $F\subseteq O$. Then, $G:=F\cup -O$ is a closed substructure of $\alg{L}^*$. Let $f:G \rightarrow \utilde{2}$ be the map defined by $f^{-1}(0)=F$. We have $f\in \cate{X}(G,\utilde{2})$ and $[F,O]=O_f$.

Conversely, let $f\in \cate{X}_p(\alg{L}^*, \utilde{M})$. Then, $f^{-1}(0)$ is a decreasing clopen subset of $\dom(f)$. Hence, it is a closed subspace of $\alg{L}^*$ and $F:=f^{-1}(0)\!\!\downarrow$ is a decreasing closed subspace of $\alg{L}^*$. Similarly, $F'=f^{-1}(1)\!\!\uparrow$ is an increasing closed subspace of $\alg{L}^*$.  It follows that
\begin{eqnarray*}
x \in O_f & \Leftrightarrow & f^{-1}(0) \subseteq x   \mbox{ and } f^{-1}(1) \subseteq -x ,\\
& \Leftrightarrow & F \subseteq x  \mbox{ and } x\subseteq -F'.
\end{eqnarray*} 
We conclude that $O_f=[F,-F']\in  \delta'$.
\end{proof}

Note that  there is no statement equivalent to Proposition \ref{prop:boils} for the variety of bounded lattices since it is not an IRF-prevariety.

\subsection{Profinite completions and products}

We aim to use the topology $\delta$ to define extension of maps between algebras of $\var{A}$ to their profinite completions. In this view, an important feature is that under our general assumptions the construction of profinite completions commutes with the construction of products, \emph{i.e.},
\begin{equation}\label{eqn:commute}
(\alg{A}\times \alg{B})^\delta\ {\simeq}\ \alg{A}^\delta \times \alg{B}^\delta,
\end{equation}
for every $\alg{A}, \alg{B}\in \var{A}$.
Given a procedure to extend maps from algebras to their profinite completions, property (\ref{eqn:commute}) would allow us to extend $n$-ary operations ($n\geq 2$) on $\alg{A}\in \var{A}$ to $n$-ary operations on $\alg{A}^\delta$. As proved in the next result, this property holds true under rather mild assumptions.

\begin{thm}\label{thm:prod}
Assume that $\alg{M}$ is {of finite type} and that $\utilde{M}$ yields a {full} duality for $\var{A}$. If $\alg{A}, \alg{B}\in\var{A}$, then
\[
(\alg{A}\times \alg{B})^\delta\  {\simeq}\ \alg{A}^\delta \times \alg{B}^\delta\quad \iff \quad (\alg{A}^* \amalg \alg{B}^*)^\flat \ \simeq\ (\alg{A}^{*})^\flat \amalg (\alg{B}^{*})^\flat.
\]
In particular, if $\utilde{M}$ yields a full logarithmic duality for $\var{A}$, then (\ref{eqn:commute}) holds for every $\alg{A}, \alg{B} \in \var{A}$ and we may assume that the isomorphism is also a $\iota$- and $\delta$-homeomorphism.
\end{thm}
\begin{proof}
Assume that $(\alg{A}^*\amalg\alg{B}^*)^\flat\simeq (\alg{A}^*)^\flat \amalg (\alg{B}^*)^\flat$. It then follows successively that
\begin{eqnarray}
(\alg{A}\times \alg{B})^\delta &  \simeq & \cate{X}^\flat ((\alg{A}\times \alg{B})^{*\flat}, \utilde{M}^\flat)\nonumber\\
& \simeq & \cate{X}^\flat ((\alg{A}^*\amalg \alg{B}^*)^{\flat}, \utilde{M}^\flat)\label{eq:mfk}\\
& \simeq & \cate{X}^\flat ((\alg{A}^*)^\flat \amalg (\alg{B}^*)^{\flat}, \utilde{M}^\flat)\label{eq:sfs}\\
& \simeq & \cate{X}^\flat((\alg{A}^*)^\flat,  \utilde{M}^\flat) \times \cate{X}^\flat((\alg{B}^*)^\flat,  \utilde{M}^\flat)\label{eq:sdgf} \\
& \simeq & \alg{A}^\delta \times \alg{B}^\delta,\nonumber
\end{eqnarray}
where (\ref{eq:mfk})
is obtained because full dualities turn products to coproducts,   (\ref{eq:sdgf}) follows from the fact that $(\var{A}_\iota(\cdot, \alg{M}_\iota), \var{X}^\flat(\cdot, \utilde{M}^\flat),e, \epsilon)$ is a dual adjunction between $\var{A}_\iota$ and $\classop{ISP}(\utilde{M}^\flat)$ and hence turns coproducts to products,  and (\ref{eq:sfs}) holds by assumption. Moreover, if  $\utilde{M}$ yields a logarithmic duality for $\var{A}$, the isomorphism given by the previous piece of argument is easily seen to be a $\iota$- and $\delta$-homeomorphism.

Conversely, if $(\alg{A}\times \alg{B})^\delta \simeq  \alg{A}^\delta \times \alg{B}^\delta$, it follows successively that
\begin{eqnarray}
(\alg{A}^* \amalg \alg{B}^*)^\flat & \simeq & (\alg{A}\times \alg{B})^{*\flat}\nonumber\\
& \simeq & \cate{A}_\iota((\alg{A}\times \alg{B})^\delta, \alg{M}_\iota)\label{eq:dqhddv}\\
& \simeq & \cate{A}_\iota(\alg{A}^\delta\times \alg{B}^\delta, \alg{M}_\iota)\label{eq:dq}\\
& \simeq & \cate{A}_\iota(\alg{A}^\delta, \alg{M}_\iota)\amalg  \cate{A}_\iota(\alg{B}^\delta, \alg{M}_\iota)\label{eq:vdki}\\
& \simeq & (\alg{A}^*)^\flat \amalg (\alg{B}^*)^\flat\label{eq:cbzt}
\end{eqnarray}
where (\ref{eq:dq}) holds by assumption, and where (\ref{eq:dqhddv}),   (\ref{eq:vdki}) and (\ref{eq:cbzt}) follow from  the fact  that  $\var{A}_\iota(\cdot, \alg{M}_\iota)$ and  $\var{X}^\flat(\cdot, \utilde{M}^\flat)$ define a dual equivalence between  $\var{A}_\iota$ and $\classop{ISP}(\utilde{M}^\flat)$ (see \cite[Theorem 2.4]{Davey2012}). 
\end{proof}

In the Appendix, we generalize Theorem \ref{thm:prod} to Boolean products.

\section{Multi-extension of maps}\label{sec:maps}
Given a map $u\colon \alg{A} \to \alg{B}$, we consider the problem of defining an extension of $u$ between $\alg{A}^\delta$ and  $\alg{B}^\delta$. Our approach (Definition \ref{defn:multi}) provides with a mutli-extension $\widetilde{u}$ of $u$, that is, a map valued in the set of (closed) subsets $\Gamma(\alg{B}^\delta_\iota)$ of $\alg{B}^\delta_\iota$. This multi-extension $\widetilde{u}$ enjoys some continuity properties (Theorem \ref{thm:cont}).

We adopt the following assumption for the remainder of the paper.
\begin{ass}\label{ass:working}
The structure $\utilde{M}$ yields a logarithmic duality for $\var{A}$ and $\utilde{M}$ is injective in $\cate{X}$.
\end{ass}
 
Surprisingly enough, as noted in \cite{NDFWA}, many known strong dualities are logarithmic and hence, satisfy Assumption \ref{ass:working}.


\begin{defn}\label{defn:multi}
Let $\alg{A}, \alg{B} \in \var{A}$ and $u\colon A \to B$. The \emph{relational extension} of $u$ is the relation $\overline{u}$ defined as the topological closure of $u$ in $\alg{A}^\delta_\delta \times \alg{B}^\delta_\iota$. The \emph{multi-extension} of $u$ is the map $\widetilde{u}$ defined on $\alg{A}^\delta$ by setting $\widetilde{u}(x)=\{y \in \alg{B}^\delta \mid (x,y)\in \overline {u}\}$ for every $x\in \alg{A}^\delta$.
\end{defn}

Let us recall that if $\struc{X, \tau}$ is a topological space, then the set $\Gamma(X)$ of closed subsets of $X$ is a complete lattice. Moreover, if $Y$ is dense in $X$,  if $C$ is a complete lattice, and if $f\colon Y  \to C$, then $\limsup_\tau f$ is the map defined on $X$   by $\limsup_\tau f(x)=\bigwedge\{\bigvee f(Y\cap U)\mid U \in\tau_x\}$. 
The next lemma shows that $\widetilde{u}$ can be computed analogously as the upper extension in the setting of bounded distributive lattices (see \cite[Definition 2.13]{Gehrke2004}).

\begin{lem}\label{lem:limsup}
Let $\alg{A}, \alg{B} \in \var{A}$ and $u\colon A \to B$. If $\widehat{u}\colon \alg{A}_\delta \to \Gamma(\alg{B}^\delta_\iota)$ is the map defined as $\widehat{u}(a)=\{u(a)\}$, then $\widetilde{u}=\limsup_\delta \widehat{u}$. 
\end{lem}
\begin{proof}
For every $x\in \alg{A}^\delta$, we have
\[
{\limsup}_\delta \widehat{u}(x)=\bigcap \{u(U\cap A)^{-} \mid  U\in\delta_x\},
\]
where the closure is computed in $\alg{B}^\delta_\iota$. It follows directly that $y\in {\limsup}_\delta \widehat{u}(x)$ if and only if $(x,y)\in \overline{u}$.
\end{proof}

Combining Proposition \ref{prop:bestcase} with Lemma \ref{lem:limsup}, and by compactness of $\alg{B}^\delta_\iota$ we obtain directly the following result.

\begin{prop}\label{prop:ext}
Let $\alg{A}, \alg{B} \in \var{A}$ and $u\colon A \to B$. 
\begin{enumerate}
\item If $a \in \alg{A}$, then $\widetilde{u}(a)=\{u(a)\}$.
\item If $x\in \alg{A}^\delta$ then $\widetilde{u}(x)$ is nonempty.
\end{enumerate}
\end{prop}

The following theorem shows that $\widetilde{u}$ enjoys similar continuity properties as the  upper extension in the setting of bounded distributive lattices (see \cite[Theorem 2.12]{Gehrke2004}).
Recall that if $\struc{X, \tau}$ is a compact \textsc{Hausdorff} space, then the family of  sets 
\[
\square U=\{F\in \Gamma(X)\mid F \subseteq U\}, \quad U  \in \tau,
\]
is a basis of a topology $\sigma{\downarrow}$, which is called the  \emph{co-\textsc{Scott} topology}.

\begin{thm}\label{thm:cont}
Let $\alg{A}$, $\alg{B}\in \var{A}$ and $u\colon A \to B$.
\begin{enumerate}
\item\label{it:lkj01} The map $\widetilde{u}\colon \alg{A}^\delta \to \Gamma(\alg{B}^\delta_\iota)$ is $(\delta, \sigma{\downarrow})$-continuous.
\item\label{it:lkj02} If $u'\colon \alg{A}^\delta \to \Gamma(\alg{B}^\delta_\iota)$ is a $(\delta, \sigma{\downarrow})$-continuous map such that $u'(a)=\{u(a)\}$ for every $a \in A$, then $\widetilde{u}(x) \subseteq u'(x)$ for every $x \in X$.
\end{enumerate}
\end{thm}
\begin{proof} First, we prove the following claim.
\smallskip

\textbf{Claim.} \emph{For any $x\in \alg{A}^\delta$ and any $F\Subset \alg{B}^*$, it holds $\widetilde{u}(x)\restriction_F~=~\bigcap\{u(V \cap A)\restriction_F \mid V \in \delta_x\}$.}

\begin{proof}[Proof of the Claim]\let\qed\relax
The inclusion $\subseteq$ is clear. 
Let us prove inclusion $\supseteq$. Let  $\alpha \in M^F$ be such that $\alpha \in u(V\cap A)\restriction_F$ for every $\delta$-neighborhood $V$ of $x$. For any finite subset $G$ of $\alg{B}^*$ that contains $F$ and any $\delta$-neighborhood $V$ of $x$, set
\[
K_{G,V}:=\{y \in \alg{B}^\delta \mid y\restriction_G\in u(V \cap A)\restriction_G\} \cap [\alpha|F].
\]
We obtain by compactness that $H_{G}:=\bigcap\{K_{G,V} \mid V \in \delta_x\}$ is a nonempty closed subspace of $\alg{B}^\delta_\iota$. It follows again by compactness that the family $\{H_G \mid F \Subset G \Subset \alg{B}^*\}$ has an nonempty intersection $H$. Any element $y$ of $H$ belongs to $\widetilde{u}(x)$ and satisfies $y\restriction_{F}=\alpha$, which proves that $\alpha \in \widetilde{u}(x)\restriction_F$.
\end{proof}

(\ref{it:lkj01}) 
%
We prove that $\widetilde{u}^{-1}(\square U)$ is an open subspace of $\alg{A}^\delta_\delta$ for any open subspace $U$ of $\alg{B}^\delta_\iota$. By compactness of $\alg{B}^\delta_\iota$, it suffices to consider the case where $U$ is a finite union of basic open sets $[\alpha|F]$ where $F\Subset \alg{B}^*$ and $\alpha \in \alg{M}^F$. We consider the case where $U$ if the union of two such basic open sets $[F_1|\alpha_1]$ and $[F_2|\alpha_2]$, as the general case can be proved in a similar way. Let $x\in \widetilde{u}^{-1}\big(\square([F_1|\alpha_1] \cup [F_2\vert \alpha_2])\big)$ and $F$ be $F_1\cup F_2$. The family $K_F:=\{u(V \cap A)\restriction_F \mid  V \in \delta_x\}$ is a downward directed family of nonempty finite sets, so it has a nonempty intersection. Let $W$ be  any $\delta$-neighborhood of $x$ such that $u(W\cap A)\restriction_F=\bigcap K_F$. We prove that $W \subseteq \widetilde{u}^{-1}\big(\square([F_1|\alpha_1] \cup [F_2\cap \alpha_2])\big)$. Let $z\in W$. 
We obtain successively
\begin{align}
\widetilde{u}(z)\restriction_F&=\bigcap\{u(V \cap A)\restriction_F \mid V\in \delta_z\}\label{eqn:fgr01}\\
& \subseteq u(W\cap A)\restriction_F\label{eqn:fgr02}\\
& = \bigcap\{u(V \cap A)\restriction_F \mid V\in \delta_x\}\label{eqn:fgr03}\\
&= \widetilde{u}(x)\restriction_F,\label{eqn:fgr04}
\end{align}
where \eqref{eqn:fgr01} and \eqref{eqn:fgr04} are obtained by the Claim, where \eqref{eqn:fgr02} holds because $W$ is a $\delta$-neighborhood of $z$, and \eqref{eqn:fgr03} holds by definition of $W$. We deduce from \eqref{eqn:fgr04} that $\widetilde{u}(z)\subseteq [F_1|\alpha_1] \cup [F_2|\alpha_2]$.

(2) By definition of the map $\widetilde{u}$, it suffices to prove that $R:=\{(x,y)\in \alg{A}^\delta_\delta\times \alg{B}^\delta_\iota \mid y \in u'(x)\}$ is a closed relation that contains $u$. We have $u\subseteq R$ by assumption. Now, let $(x,y)\in \alg{A}^\delta_\delta\times \alg{B}^\delta_\iota$ such that $y\not\in u'(x)$. Since $u'(x)$ is a closed subspace of the Boolean space $\alg{B}^\delta_\iota$, there is a clopen subspace $U$ of $\alg{B}^\delta_\iota$ such that $z\not\in U$ and $u'(x)\subseteq U$. It follows by continuity of $u'$ that $u'^{-1}(\square U)\times (\alg{B}^\delta_\iota \setminus U)$ is a neighborhood of $(x,y)$ disjoint from $R$.
\end{proof}
\begin{rem}\label{rem:de}
It follows from the Claim stated in the proof of Theorem \ref{thm:cont} that if $u\colon \alg{A} \to \alg{B}$ then $\widetilde{u}(x)\restriction_{\{\phi\}}=\widetilde{\phi \circ u}(x)$ for any $x\in \alg{A}^\delta$ and any $\phi \in  \alg{B}^*$.
\end{rem}

\section{Coninuity properties and function compositions}\label{sec:compo}

Theorem \ref{thm:cont} characterizes $\widetilde{u}$ as the smallest $(\delta, \sigma{\downarrow})$-continuous extension $u'\colon \alg{A}^\delta \to \Gamma(\alg{B}^\delta_\iota)$ of $u$. In this section, we investigate the properties of $\widetilde{u}$ under additional continuity assumptions.

\subsection{Smoothness and strongness}

The case where the relational extension $\overline{u}$ of  $u\colon \alg{A} \to \alg{B}$ is a function  leads us  to the following natural definition.
\begin{defn}\label{defn:smo}
Let $\alg{A}, \alg{B} \in \var{A}$ and $u\colon\alg{A} \rightarrow \alg{B}$. We say that $u$ is is \emph{smooth} if $\overline{u}$ is a function, that is, if $\widetilde{u}(x)$ is a one-element set for every $x\in \alg{A}^\delta$. In this case, we denote by $u^\delta$ the map $u^\delta:\alg{A}^\delta \rightarrow \alg{B}^\delta$ defined by $u^\delta(x)\in \widetilde{u}(x)$.
\end{defn}
\begin{ex}\label{ex:smooth00}
If $\iota \subseteq \delta$ then any term function is smooth since it is $(\iota, \iota)$-continuous.
\end{ex}
 Theorem \ref{thm:cont} can be rephrased for smooth maps in the following way.
\begin{prop}\label{prop:smooth}
Let $\alg{A}, \alg{B} \in \var{A}$ and $u:\alg{A} \rightarrow \alg{B}$. 
\begin{enumerate}
\item\label{it:uyt01} If $u$ is smooth then $u^\delta:\alg{A}^\delta\rightarrow \alg{B}^\delta$ is a $(\delta,\iota)$-continuous extension of $u$.
\item\label{it:uyt02} If $u$ admits a $(\delta,\iota)$-continuous extension $u':\alg{A}^\delta \rightarrow \alg{B}^\delta$ then $u$ is smooth and $u^\delta=u'$.
\end{enumerate}
\end{prop}
Propositions \ref{prop:boils}  and \ref{prop:smooth} show that the notion of {smoothness} as defined in Definition \ref{defn:smo} boils down to the one defined in \cite{Geh00} when it is considered for the variety of bounded distributive lattices.
As a corollary of Proposition \ref{prop:smooth}(2), we obtain that if $u\colon \alg{A} \to \alg{B}$ is not smooth, then it is not even possible to define a continuous extension $u^\delta\colon \alg{A}^\delta_\delta \to \alg{B}^\delta_\iota$ by suitably picking up an element $u^\delta(x)$ in $\widetilde{u}(x)$ for every $x\in \alg{A}^\delta$.
\begin{ex}\label{ex:smooth01} \emph{If $\iota \subseteq \delta$, then every element $\phi \in \alg{A}^*$ is smooth.}  Consider the map $\phi'\colon \alg{A}^\delta \to M$ defined by $\phi'(x)=x(\phi)$. For any $F\subseteq M$, we have ${\phi'}^{-1}(F)=\bigcup_{f\in F}[\phi:f] \cap \alg{A}^\delta$, which proves that $\phi'$ is $(\iota, \iota)$-continuous, so it is $(\delta, \iota)$-continuous since $\delta\subseteq  \iota$. The conclusion follows from Proposition \ref{prop:smooth} (\ref{it:uyt02}).
\end{ex}

\begin{prop} \label{prop:local}
Let $\alg{A}, \alg{B} \in \var{A}$ and $u:\alg{A} \rightarrow \alg{B}$. The map $u$ is smooth if and only if $\phi \circ u$ is smooth for every $\phi \in \alg{B}^\delta$.
\end{prop}
\begin{proof}
If $u$ is smooth and $\phi \in \alg{B}^*$, then the map $\phi'\circ u$ where $\phi'(x):=x(\phi)$ for any $x\in \alg{B}^\delta$ is a $(\delta, \iota)$-continuous extension of $\phi \circ u$. It follows that $\phi \circ u$ is smooth by Proposition \ref{prop:smooth}(\ref{it:uyt02}).
Conversely, assume that $\phi\circ u$ is smooth for every $\phi \in \alg{B}^*$. We prove that the map $u'\colon \alg{A}^\delta_\delta \to \alg{B}^\delta_\iota$ defined as $u'(x)=(\phi \circ u)^\delta \circ x$ is continuous, and the conclusion follows from Proposition  \ref{prop:smooth}(\ref{it:uyt02}). Let $F$ be a finite subset of $\alg{B}^*$ and $\alpha \in M^F$. We have 
\[u'^{-1}([F|\alpha])=\bigcap\{\big((\phi \circ u)^{\delta}\big)^{ -1}(\{\alpha(\phi)\}) \mid \phi \in F\},
\]
which proves that $u'$ is $(\delta, \iota)$-continuous by  Proposition  \ref{prop:smooth}(\ref{it:uyt02}) and our assumption.
\end{proof}
\begin{ex} \emph{If $\iota \subseteq \delta$, then every $u\in \var{A}(\alg{A}, \alg{B})$ is smooth.} This result follows from  Example \ref{ex:smooth01} and Proposition \ref{prop:local}. If $\alg{M}$ is of finite type, it can also be considered as a consequence of \cite[Theorem 2.4]{Davey2012}.
\end{ex}

\begin{defn}\label{defn:strong}
Let $\alg{A}, \alg{B} \in \var{A}$ and $u\colon\alg{A} \rightarrow \alg{B}$. We say that $u$ is is \emph{strong} if   $\widetilde{u}$ is $(\iota, \sigma{\downarrow})$-continuous.
\end{defn}

The proof of the following Lemma is straightforward.

\begin{lem}\label{lem:strongg}
Assume that $\iota \subseteq \delta$, and let $u\colon \alg{A} \to \alg{B}$ be a smooth map. Then $u$ is strong if and only if $u^\delta$ is $(\iota, \iota)$-continuous. 
\end{lem}

\begin{ex}\label{ex:stronggg}\emph{If $\iota \subseteq \delta$, then every $u\in \var{A}(\alg{A}, \alg{B})$ is strong.} We already know that $u$ is smooth, and we prove as in Example \ref{ex:smooth01} that $u^\delta(x)(\phi)=x(\phi \circ u)$ for every $x\in \alg{A}^\delta$ and $\phi \in  \alg{B}^*$. It follows that $u^\delta$ is $(\iota, \iota)$-continuous, or equivalently, that $\widetilde{u}$ is $(\iota, \sigma{\downarrow})$-continuous by Lemma \ref{lem:strongg}.
\end{ex}

Strongness can be used to obtain the preservation of functional composition through profinite completions, as illustrated in the next proposition.

\begin{prop}\label{prop:strong}
Let $\alg{A}, \alg{B}, \alg{C} \in \var{A}$, $u:\alg{A}\rightarrow \alg{B}$ and $v:\alg{B}\rightarrow \alg{C}$. 
\begin{enumerate}
\item\label{it:mfr01} If $v$ is strong then $\overline{vu}\subseteq \overline{v}\circ\overline{u}$.
\item\label{it:mfr02} If $u$ is smooth and if $v$ is strong and smooth, then $vu$ is smooth and $(vu)^\delta=v^\delta u^\delta$.
\end{enumerate}
\end{prop}
\begin{proof}
First, we prove the following claim.
\smallskip

\noindent {\bf Claim.} \emph{For any strong map $u\colon\alg{A}\to\alg{B}$, the map $\widecheck{u}\colon \Gamma(\alg{A}^\delta_\iota) \to \Gamma(\alg{B}^\delta_\iota)$ defined by $\widecheck{u}(K)=\bigcup\widetilde{u}(K)$ is $(\sigma{\downarrow}, \sigma{\downarrow})$-continuous.}

\begin{proof}[Proof of the Claim.]\let\qed\relax
First, we prove that $\widecheck{u}(K)$ is a closed subspace of $\alg{B}^\delta_\iota$ for every closed subspace $K$ of $\alg{A}^\delta_\iota$. Let $y\in \alg{B}^\delta_\iota$ such that $y\not\in\widecheck{u}(K)$. For every $x\in K$ let $V_x$ and $W_x$ be disjoint $\iota$-neighborhood of $\widetilde{u}(x)$ and $y$, respectively. By continuity of $\widetilde{u}$ and compactness of $K$, there is a finite subset $F$ of $K$ such that $\{\widetilde{u}^{-1}(\square V_x) \mid x \in F\}$ covers $K$. It follows that $W:=\bigcap\{W_x \mid x \in F\}$ is a $\iota$-neighborhood of $y$ that does not meet $\widecheck{u}(K)$. 

Now, for any open subspace $U$ of $\alg{B}^\delta_\iota$, it is not difficult to prove that \[\widecheck{u}^{-1}(\square U)=\square\widetilde{u}^{-1}(\square U).\]
The continuity of $\widecheck{u}$ follows from the latter identity and strongness of $u$.
\end{proof}

(\ref{it:mfr01}) By the Claim, the map $\widecheck{v}\widetilde{u}$ is a $(\iota,\sigma{\downarrow})$-continuous extension of $vu$. Then, we obtain (\ref{it:mfr01}) by  Theorem \ref{thm:cont}.

(\ref{it:mfr02}) By the Claim, the function $w\colon \alg{A}^\delta_\delta \to \alg{C}^\delta_\iota$  that maps every $x\in \alg{A}^\delta$ to the only element of $\widecheck{v}\widetilde{u}(x)$ is continuous. Then, we obtain (\ref{it:mfr02}) by Proposition \ref{prop:smooth}.
\end{proof}

\begin{cor}
Assume that $\iota \subseteq \delta$ and let $\alg{A}, \alg{B}, \alg{C} \in \var{A}$.
\begin{enumerate}
\item\label{it:sde01} Any term function $u:=t^{\alg{A}}(s_1^{\alg{A}}, \ldots, s_\ell^{\alg{A}})$ is smooth and strong, and  $u^\delta=(t^{\alg{A}})^\delta((s_1^{\alg{A}})^\delta, \ldots, (s_\ell^{\alg{A}})^\delta)$.

\item\label{it:sde02} If $u \in \var{A}(\alg{A}, \alg{B})$ and $v \in \var{A}(\alg{B}, \alg{C})$, then $vu$ is smooth and strong and $(vu)^\delta=u^\delta v^\delta$.
\end{enumerate}
\end{cor}
\begin{proof}
(\ref{it:sde01}) The proof is obtained by induction on the construction of the term using Example \ref{ex:smooth00}, Lemma \ref{lem:strongg} and Proposition \ref{prop:strong} (\ref{it:mfr02}) since any term function is $(\iota, \iota)$-continuous.

(\ref{it:sde02}) The proof is an application of Lemma \ref{lem:strongg}, Example \ref{ex:stronggg}, and Proposition \ref{prop:strong}~(\ref{it:mfr02}).
\end{proof}

\subsection{A sample case: profinite completions of median algebras}
In this subsection, we illustrate the previous constructions by considering  that $\var{A}$ is the variety of median algebras, that is, $\var{A}=\classop{ISP}(\alg{2})$ where $\alg{2}=\struc{\{0,1\}, \me }$ is the algebra with a single ternary operation $\me$ defined as the majority function on $\{0,1\}$. This variety is of special interest as  (i) it is  not lattice-based, (ii) it admits a strongly logarithmic duality, and (iii) the dual category is locally finite. Hence,  $\iota(\alg{A}^\delta)\subseteq\delta(\alg{A}^\delta)$ for every $\alg{A}\in \var{A}$.

\subsubsection{A natural duality for median algebras}
It is known \cite{NDFWA,Isbell1980, Werner1981} that the topological structure 
\[\utilde{2}=\struc{\{0,1\}, 0,1, \leq, {}^\bullet, \iota}\]
 with two constants $0$ and $1$, the natural order $\leq$, and the unary operation ${}^\bullet$ defined by $x^\bullet\equiv (x+1)\mod 2$, yields a strong logarithmic duality for $\var{A}$. A topological structure $\str{X}=\struc{X, 0, 1, \leq, {}^\bullet, \tau}$ is a object of the dual category $\cate{X}=\classop{IS}_c\classop{P}(\utilde{2})$ provided that $\struc{X, \leq, \iota}$ is a \textsc{Priestley} space with bounds $0$ and $1$, that ${}^\bullet$ is an order reversing homeomorphism that swaps $0$ and $1$ and that satisfies $\phi^{\bullet\bullet}=\phi$, and $\phi \not\leq \phi^{\bullet}$ for every $\phi \neq 0$.

There is an equivalent spectrum-based formulation of this duality that eases computations. A subset $\phi$ of a median algebra $\alg{A}$ is  \emph{prime convex} if for every $x, y, z \in \alg{A}$, the element $\me(x,y,z)$ belongs to $\phi$ if and only if at least one of the sets $\{x, y\}$, $\{x,z\}$, $\{z,y\}$ is a subset of $\phi$. A subset $x$ of a structure $\str{X}\in \cate{X}$ is a \emph{disjoint ideal} of $\str{X}$ if it is a downset set disjoint with $x^\bullet$. If in addition $x$ is a clopen subset of $\str{X}$, then $x$ is called a \emph{continuous disjoint ideal}. A \emph{(continuous) maximal disjoint ideal} of $X$ is a (continuous) ideal that contains $\phi$ or $\phi^\bullet$ for every $\phi\in X.$

It is not difficult to show that the map $\phi \mapsto \phi^{-1}(0)$ is an isomorphism between $\alg{A}^*$ and  the prime spectrum of $\alg{A}$ (\emph{i.e.}, the set of prime convex subsets of $\alg{A}$) equipped with inclusion order, $\emptyset$ and $\alg{A}$ as bottom and top element respectively, set complementation as map ${}^\bullet$, and \textsc{Zariski} topology. If $\str{X}\in\cate{X}$, then the dual $X_*$ of $\str{X}$ is isomorphic to  the set of continuous maximal disjoint ideals of $\str{X}$ equipped with the operation $\me$ inherited from the median operation defined on the powerset of $X$ as
\begin{equation}\label{eqn:ope}
 \me(x, y ,z)=(x\cap y)\cup (x\cap z)\cup(y\cap z).
\end{equation}
If $\alg{A}\in \var{A}$, then $\alg{A}^\delta$ is isomorphic to the set of the maximal disjoint ideals of $\alg{A}^*$ equipped with the operation  defined in (\ref{eqn:ope}).

\subsubsection{Profinite completions of Boolean powers of $\alg{2}$}
We can apply Theorem \ref{thm:bool_prod} to compute profinite completions of Boolean powers of the median algebra $\alg{2}$.
\begin{prop}\label{prop:bool}
 If $\alg{A}$ is a median algebra that has a Boolean representation  $\alg{A}\hookrightarrow \alg{2}^X$, then  $\alg{A}^\delta_\iota$ is isomorphic (algebraically and topologically) to $\alg{2}_\iota^X$.
\end{prop}
\begin{proof}
The dual of $\alg{2}$ is depicted in Figure \ref{fig:dual2}. Observe that for every nonempty finite sets $I$ and $J$, every $a\in \alg{2}^I$ and $b\in \alg{2}^J$, identity \[\alg{2}^*=\bigcup_{i\in I}[a_i:1] \cup \bigcup_{j\in J}[b_j:0]\] holds if and only if $\bigcap_{j\in J}[b_j:1] \subseteq \bigcup_{i\in I}[a_i:1]$, that is, if and only if the following condition is satisfied in $\alg{2}$ (for some $j_0\in J$),
\[
\bigwedge_{j\in J} (b_j=b_{j_0}) \Rightarrow  \bigvee_{i\in I} (a_i=b_{j_0}).
\]  
The latter formula is also equivalent to
\[
\bigvee_{k,l\in J; \, i\in I} \big(\me(a_i,b_k,b_l)=a_i\big).
\]
We conclude the proof by applying  Theorem \ref{thm:bool_prod}.
\begin{figure}
\begin{tikzpicture}[line cap=round,line join=round,x=1.0cm,y=1.0cm, scale=0.6]
\draw (2,4)-- (0,2);
\draw (2,0)-- (0,2);
\draw (2,0)-- (4,2);
\draw (2,4)-- (4,2);
\draw [rotate around={-45:(1,1)}] (1,1) ellipse (1.51cm and 0.53cm);
\draw [rotate around={45:(3,1)}] (3,1) ellipse (1.53cm and 0.57cm);
\draw [rotate around={-45:(3,3)},dotted] (3,3) ellipse (1.52cm and 0.55cm);
\draw [rotate around={45:(1,3)},dotted] (1,3) ellipse (1.53cm and 0.58cm);
\begin{scriptsize}
\fill  (2,4) circle (1.5pt);
\draw (2.16,4.46) node {$\alg{2}$};
\fill  (0,2) circle (1.5pt);
\draw (-0.20,2.26) node[left] {$\{0\}$};
\fill  (2,0) circle (1.5pt);
\draw (2.16,-0.2) node[below] {$\emptyset$};
\fill  (4,2) circle (1.5pt);
\draw (4.16,2.26) node[right] {$\{1\}$};
\draw (0.5,0.5) node[left] {$[1:1]$};
\draw (3.5,0.5) node[right] {$[0:1]$};
\draw (0.5,3.5) node[left] {$[0:0]$};
\draw (3.5,3.5) node[right] {$[1:0]$};
\end{scriptsize}
\end{tikzpicture}
\caption{Dual of median algebra $\alg{2}$}
\label{fig:dual2}
\end{figure}
\end{proof}
Corollary \ref{cor:surpr}  is a surprising consequenc of Proposition \ref{prop:bool}.  We say that a median algebra $\alg{A}=\struc{A,  \me}$ is a \emph{Boolean} if there is a Boolean algebra $\struc{A, \vee, \wedge, \neg, 0,1}$ such that $\me(x,y,z)=(x\wedge y) \vee (x\wedge z) \vee (y\wedge z)$ for every $x, y, z$ in $A$.
Recall that an algebra $\alg{A}=\struc{A, \me, \cdot^c}$ of type $(3,1)$ is a  \emph{ternary Boolean algebra} \cite{Grau1947} if $\struc{A, \me}$ is a median algebra and the equation $\me(x,z,x^c)=z$ holds in $\alg{A}$. 

\begin{lem}[{{\cite{Grau1947}}}]\label{lem:grau}
A median algebra is Boolean if and only if it is the $\{\me\}$-reduct of a ternary Boolean algebra.
\end{lem}

\begin{cor}\label{cor:surpr}
Let $\alg{A}$ be a median algebra. The following conditions are equivalent.
\begin{enumerate}[(i)]
\item\label{it:jhg01} $\alg{A}^\delta$ is Boolean.
\item\label{it:jhg02} $\alg{A}$ is a Boolean power of $\alg{2}$.
\end{enumerate}
\end{cor}
\begin{proof}
(\ref{it:jhg01}) $\implies$ (\ref{it:jhg02}) Let $\alg{A}$ be a Boolean median algebra, and let $\alg{A}^\flat$ be a Boolean algebra whose $\{\me\}$-reduct is $\alg{A}$. Then $\alg{A}^\flat$ can be represented as a Boolean power $\alg{A}^\flat \hookrightarrow \alg{2}^{X}$, where $X$ is the Stone dual of $\alg{A}^\flat$.  This Boolean representation still holds between the $\{\me\}$-reducts of $\alg{A}^\flat$ and $\alg{2}^X$.

(\ref{it:jhg02}) $\implies$ (\ref{it:jhg01}) We know by Proposition \ref{prop:bool} that we can identify $\alg{A}^\delta$ with $\alg{2}^X$. Denote by $\cdot^c$ the operation defined on $\alg{2}^X$ by
\[
x^c(\phi)\equiv 1+x(\phi) \mod 2, \qquad \phi \in {X}.
\]
Then $\struc{2^X, \me, \cdot^c}$ is a ternary Boolean algebra, and we conclude the proof by Lemma \ref{lem:grau}
\end{proof}
We conclude the section by giving an example of a smooth function which is not a homomorphism.

\begin{ex}
In the $\wedge$-semillatice $\struc{A, \leq}$ depicted in Fig. \ref{fig:graph},  any three elements have an upper-bound whenever each pair of them is bounded above, and any principal ideal is a distributive lattice. Hence, it is a median semilattice \cite{Sholander1954}. It follows that the operation $\me$ defined on $A$ as $\me(x,y,z)=(x\wedge y)\vee(x\wedge z)\vee (y\wedge z)$ turns $A$ into a median algebra $\alg{A}$. This operation can be easily computed explicitly: for every $j,k,\ell \in \omega$
\[
\begin{array}{ll}
\me(a_j, a_k, a_\ell)=a_{(j,k,\ell)} \quad & \me(a_j, a_k, b_\ell)=a_{(j,k,\ell)}\\
 \me(a_j, b_k, b_\ell)=a_{(j,k,\ell)} \quad & \me(b_j, b_k, b_\ell)=a_{(j,k,\ell)},
\end{array}
\]
where $(j,k,\ell)$ denotes the median element of $j,k,\ell\in \omega$.
\begin{figure}
\begin{center}
\begin{tikzpicture}

\coordinate (A) at (0,0);
\fill (A) circle(0.057) node[left]{$a_0$};
\fill (A)++(0,0.5) circle(0.057);
\fill (A)++(0,1) circle(0.057);
\fill (A)++(0,1.5) circle(0.057);
\coordinate (B) at (0.5,0.5);
\fill (B) circle(0.057);
\fill (B)++(0,0.5) circle(0.057);
\fill (B)++(0,1) circle(0.057);
\draw(A) -- ++(0,0.5) node[left]{$a_1$};
\draw (A)++(0,0.5)-- ++(0,0.5)node[left]{$a_2$} ;
\draw (A) ++(0,1)-- ++(0,0.5) node[left]{$a_3$};

\draw (A) -- (B) node[right]{$b_0$};
\draw (A)++(0,0.5) -- ++(0.5,0.5) node[right]{$b_1$};
\draw (A)++(0,1) -- ++(0.5,0.5) node[right]{$b_2$};
\draw[dashed] (A)++(0,1.5) -- ++(0,0.5);
\draw[dashed] (A)++(0.5,1.5) -- ++(0,0.5);
\end{tikzpicture}
\end{center}
\caption{Graph of median algebra $\alg{A}$}
\label{fig:graph}
\end{figure}

Clearly, the elements of $\alg{A}^*$ are 
\[
A_i=\uperset{a_i},\quad A_i^\bullet=A\setminus\uperset{a_i}, \quad B_i=\{b_i\}, \quad B_i^\bullet=A\setminus \{b_i\}, \qquad i \in \omega.
\]
Hence, the dual of $\alg{A}$ is depicted in Fig. \ref{fig:dual}.

The elements of the bidual of $\alg{A}$ are easily computed:
\[
e_{\alg{A}}(b_n)=\lowerset{A_{n+1}} \cup \lowerset{A_n^\bullet}=\lowerset{B_n^\bullet}, \qquad n \in \omega,
\]
\[
e_{\alg{A}}(a_n)=\lowerset{A_{n+1}} \cup \lowerset{A_n^\bullet} \cup \{B_n\}, \qquad n \in \omega. 
\]
Then $\alg{A}^\delta \setminus e_{\alg{A}}(\alg{A})=\{\infty\}$ where 
\[
\infty=\bigcup\{\{A_n^\bullet, B_n\}\mid n \in \omega\}.
\]
A simple computation shows that, up to identification of $\alg{A}$ with $e_{\alg{A}}(\alg{A})$
\[
\me(\infty, a_m, b_n)=\me(\infty, a_m, a_n)=\me(\infty, b_m, b_n)=a_{m \vee n}, \qquad m, n \in \omega.
\]

Let us illustrate the inclusion $\iota \subseteq \delta$.  For any $\phi \in  \alg{A}^*$, the subasis clopen subsets $\{x \mid  \phi \in x\}$ and $\{x \mid  \phi \not\in x\}$ of $\alg{A}^\delta_\iota$ are respectively equal to $O_{f}$ and $O_g$ where $f=\{\phi\}$ and $g=\{\phi^\bullet\}$  correspond to morphisms defined on the closed substructure $\{\phi, \phi^\bullet\}$ of $\alg{A}^*$.

Now, let $u\colon\alg{A}\rightarrow \alg{2}$ be the map defined by $u(b_i)=1$ and $u(a_i)=0$ for any $i \in \omega$. Clearly, the map $u$ is not a median homomorphism (neither a $\wedge$-homomorphism). Let us denote by $u'$ the extension of $u$ on $\alg{A}^\delta$ that satisfies $u'(\infty)=0$. We prove that $u'$ is $(\delta, \iota)$-continuous which implies that $u$ is smooth by Proposition \ref{prop:smooth}. We have to prove that $u'^{-1}(0)=\{\infty, a_0, a_1, \ldots\}$ is a $\delta$-open subset of $\alg{A}^\delta$. Consider $K=\{\emptyset, B_0, B_1, B_2, \ldots\}=\bigcap_{i\in \omega}e_{\alg{A}}(a_i)$. It follows from the continuity of ${}^\bullet$ that $K\cup K^\bullet$ is a closed substructure of $\alg{A}^*$. Hence, the map $f:K\cup K^\bullet\rightarrow \str{2}$ defined by $f(x)=0$ if and only if $x \in K$ is a partial morphism on $\alg{A}^*$. It is easily seen that $\infty \in O_f \subseteq u'^{-1}(0)$.
\begin{figure}
\begin{center}

\begin{tikzpicture}
\draw (3,5) node[above]{$A$};
\draw (3,5)-- (3.65,4.35) node[below right]{$B_1^\bullet$};
\draw[dashed] (3,5) -- (1.66, 4.46);
\draw (3,5)-- (3,4.35) node[below]{$B_2^\bullet$};
\draw[dashed] (2.87,3.87)--(2.45,3.45);
\draw (5,3.8)-- (5,3.41);
\draw (4.1,3.93)-- (4.87,3.41);
\draw (5,3.2) node{$A_2$};
\draw (5,2.9)--(5,2.7)node[below]{$A_3$};
\draw (3,3.87)-- (4.91,2.57);
\draw(3,0)-- (2.11,0.89);
\draw (2,1) node{$B_1$};
\draw (3,0)-- (3,0.8);
\draw  (3,1) node{$B_2$};

\draw (1,2.2) -- (1,1.9);
\draw (1,2.36) node{$A_3^\bullet$};
\draw[dashed] (1,2.5)--(1,2.95);
\draw (1,1.55)   -- (1,1.15);
\draw (1,1.68)  node{$A_2^\bullet$};
\draw (1.83,1.12) -- (1.22,1.53);
\draw (2.83,1.12)-- (1.22,2.21);
\draw (2.12,1.12)-- (4.94,3.94);
\draw (5,4) node{$A_1$};
\draw (3.12,1.12)-- (4.94, 2.94);
\draw (3.93,3.93)-- (1.11,1.11);
\draw (4.2,4.52) -- (3,5);
\draw (4.34,4.46) node{$B_0^\bullet$};
\draw (4.46,4.38)-- (4.89,4.08);
\draw (3,0) node[below]{$\emptyset$}-- (1.8,0.46);
\draw (1.7,0.51) node{$B_0$};
\draw[dashed] (3,0) -- (4.3,0.51);
\draw (1.57,0.6)-- (1.12,0.91) ;
\draw (1,1) node{$A_1^\bullet$};
\end{tikzpicture}
\caption{The dual $\alg{A}^*$ of $\alg{A}$.}
\label{fig:dual}
\end{center}
\end{figure}
\end{ex}

\section{Extensions of functions in ordered setting}\label{sec:order}
Given a map $u\colon\alg{A} \to \alg{B}$, Definition \ref{defn:multi} provides a relation (or a multi-map) $\bar{u}\subseteq {\alg{A}_\delta^\delta}\times \alg{B}^\delta_\iota$ that extends $u$. 
In the case of bounded (distributive) lattices  $\alg{A}$ and $\alg{B}$,  the classical technique  \cite{GehHar,Gehrke2004}  adopted to extend $u$  to the  canonical extensions (\emph{i.e.}, profinite completions) of $\alg{A}$ and $\alg{B}$ provides two functions: the lower extension $u^\sigma$ and the upper extension $u^\pi$.  In this section, we  reconcile these two  approaches and prove that in the context of bounded distributive lattices, the multi-extension $\widetilde{u}$ enables us to recover $u^\sigma$ and $u^\pi$, but not conversely. Our approach leads to more general results about varieties of algebras that are $\iota$-locally semilattices (Definition \ref{defn:localsem}).

%

\begin{nota}
Let $\leq$ be a fixed total order on ${M}$. We denote by $\iota{\uparrow}$, respectively $\iota{\downarrow}$, the topologies formed by the upsets, respectively the downsets, of $(M, \leq)$. 
\end{nota}


We can use the total order $\leq$ defined on $M$ to construct  an upper and a lower extension of any map $u\colon \alg{A} \to \alg{B}$.

\begin{defn}\label{defn:localsem}
Let $\alg{A}, \alg{B} \in \var{A}$ and $u:\alg{A}\rightarrow \alg{B}$. We define the maps  $u^\vartriangle\colon \alg{A}^\delta \to M^{\alg{B}^*}$ and $u^\triangledown\colon \alg{A}^\delta \to M^{\alg{B}^*}$ by
\[
u^\triangledown(x) :=\bigwedge \widetilde{u}(x), \qquad u^\vartriangle(x) :=\bigvee \widetilde{u}(x),
\]
for every $x\in \alg{A}^\delta$.
We call $u^\vartriangle$ \emph{the upper extension of $u$}, and $u^\triangledown$ the  \emph{the lower extension of $u$}.
\end{defn}

\begin{lem} \label{lem:00}
If $\alg{A}, \alg{B} \in  \var{A}$ and $u:\alg{A}\rightarrow \alg{B}$, then   $u^\triangledown(x)(\phi)=\bigwedge \widetilde{(\phi \circ u)}(x)$ and $u^\vartriangle(x)(\phi)=\bigvee \widetilde{(\phi \circ u)}(x)$. 
\end{lem}
\begin{proof}
The proof follows from Remark \ref{rem:de}.
\end{proof}

 Theorem \ref{thm:uplowetx} gives sufficient conditions for $u^\vartriangle$ and $u^\triangledown$  to be valued in $\alg{B}^\delta$.
\begin{defn}
An algebra $\alg{A}\in \var{A}$ is a \emph{local meet-semilattice} if for every $b, c \in \alg{A}$ and every $F\Subset \alg{A}^*$,  it holds
$(b\wedge c)\restriction_F \in \alg{A}\restriction_F.$
\emph{Local join-semilattices} are defined dually. A \emph{local lattice} is an algebra of $\var{A}$ that is both a local meet-semilattice and a local join-semilattice.
\end{defn}

\begin{thm}\label{thm:uplowetx}
Let $\alg{A}, \alg{B} \in  \var{A}$ and $u:\alg{A}\rightarrow \alg{B}$. 
\begin{enumerate}
\item\label{it03:01} The map $u^\triangledown\colon \alg{A}^\delta \to M^{\alg{B}^*}$ is a $(\delta, \iota{\uparrow})$-continuous extension of $u$.
\item\label{it03:01bis} The map $u^\vartriangle\colon \alg{A}^\delta \to M^{\alg{B}^*}$ is a  $(\delta, \iota{\downarrow})$-continuous extension of $u$.
 \item\label{it03:04} If $\alg{B}$ is a local meet-semilattice, then the map $u^\triangledown$ is valued in $\alg{B}^\delta$.
\item\label{it03:05} If $\alg{B}$ is a local join-semilattice, then the map $u^\vartriangle$ is valued in $\alg{B}^\delta$.
\item\label{it03:06} If $\alg{B}$ is a local lattice, then $u^\triangledown$ and $u^\vartriangle$ are valued in $\alg{B}^\delta$.
\end{enumerate}
\end{thm}
\begin{proof}

(\ref{it03:01}) Let $F\Subset \alg{B}^*$ and $\alpha\in M^F$. We obtain by Lemma \ref{lem:00} that 
\[(u^\triangledown)^{-1}([F|\geq \alpha])=\bigcap\{\widetilde{\phi\circ u}^{-1}(\square U_\phi)\mid \phi \in F\},\] where $U_\phi:=\alpha(\phi){\uparrow}$  for every $\phi \in F$. Then, the continuity of $u^\triangledown$ follows from Theorem \ref{thm:cont}. The fact that $u^\triangledown$ is an extension of $u$ follows from Proposition \ref{prop:ext}.

(\ref{it03:01bis}) is obtained from (\ref{it03:01})  by duality.

(\ref{it03:04}) Let $x$ be an element of $\alg{A}^\delta$ and let us prove that $u^\triangledown(x)$ is in the closure of $\alg{B}$ in $\alg{M}_\iota^{\alg{B}^*}$. We proceed as in the proof of Theorem \ref{thm:cont} and for any $F \Subset \alg{B}^*$ we choose a $\delta$-neighborhood $W$ of $x$ such that $\widetilde{u}(x)\restriction_F=u(W\cap A)\restriction_F$. The family $u(W\cap A)\restriction_F$ is finite, and since $\alg{B}$ is a local meet-semilattice, there is some $c\in \alg{B}$ such that $u^\triangledown(x)\restriction_F=\bigwedge u(W\cap A)\restriction_F=c\restriction_F$. We have proved that any  $\delta$-neighborhood of $x$ meets $\alg{B}$.

(\ref{it03:05}) is obtained from (\ref{it03:04})  by duality, and (\ref{it03:06}) follows from (\ref{it03:04}) and  (\ref{it03:05}) .
\end{proof}
Given a map $u\colon \alg{L} \to \alg{L}'$ between two bounded distributive lattices $\alg{L}$ and $\alg{L}'$, the theory of canonical extension \cite{Gehrke2004} provides with the \emph{upper extension} $u^\pi\colon \alg{L}^\delta \to \alg{L'}^\delta$ and the \emph{lower extension} $u^\sigma\colon \alg{L}^\delta \to \alg{L'}^\delta$. The following corollary proves that they can be recovered from the multi-extension $\widetilde{u}$ of $u$. 

\begin{cor}\label{cor:lat}
If $\alg{L}$ and $\alg{L'}$  are two bounded distributive lattices and $u:\alg{L} \to \alg{L}'$, then for every $x\in \alg{L}$ it holds $u^\sigma=u^\triangledown$ and $u^\pi=u^\vartriangle$. 
\end{cor}
\begin{proof}
The proof follows from the application of Theorem \ref{thm:uplowetx}  to the variety $\var{A}$ of bounded distributive lattices with $M=\{0,1\}$ ordered in the natural way.
\end{proof}

\begin{ex}

Let $\alg{L}$ be the bounded distributive lattice made of $\omega$ and the finite subsets of $\omega$ with inclusion order. The \textsc{Priestley} dual $\alg{L}^*=\omega\cup\{\infty\}$ is the one point \textsc{Alexandroff} compactification of the antichain $\omega$, with $\infty$ as top element. Hence, $\alg{L}^\delta=2^\omega\cup \{\top\}$ is the power set of $\alg{L}^*\setminus\{\infty\}$ with an additional top element $\top=\alg{L}^*$.
\begin{enumerate}
\item We easily build functions $u:\alg{L}\rightarrow \alg{2}$ that are smooth without being homomorphisms. Indeed let $u$ be the non trivial permutation of $\alg{2}$ and $\phi\in \alg{L}^*$. Then $u\circ\phi:\alg{L}\rightarrow \alg{2}$ is a smooth function that does not belong to $\alg{L}^*$. Other examples are given by the maps $u_A:\alg{L}\rightarrow \alg{2}$ (for $A \subseteq \omega$) that are defined by $u_A(x)=0$ if and only if $x\subseteq A$. If $A$ is infinite and co-infinite then $u_A$ is smooth but not strong.
\item\label{ex:end02} The function $u\colon\alg{L}\rightarrow \alg{2}$ defined by $u(X)=\card{X}\mod 2$ if $X\neq \omega$ and $u(\omega)=1$ is not smooth. Indeed, if $X$ is an infinite proper subset of $\omega$ then $\widetilde{u}(X)=\{0,1\}=[u^\triangledown(x),u^\vartriangle(x)]$.
\item The function $u:\alg{L}\rightarrow \alg{2}^2$ defined by 
\[u(X)=(\card{X} \mod 2, (\card{X}+1)\mod 2), \quad X\neq \omega,
\]
\[u(\omega)=(1,1),\]
is not smooth. Moreover, contrary to example (\ref{ex:end02}), the set $\widetilde{u}(x)$ is not determined by $u^\triangledown(x)$ and $u^\vartriangle(x)$. Indeed, if $X$ is an infinite proper subset of $\omega$ then $\widetilde{u}(X)=\{(0,1), (1,0)\}$ while $u^\triangledown(x)=(0,0)$ and $u^\vartriangle(x)=(1,1)$. 
\item For $k\geq 2$ let $u_k:\alg{L}\rightarrow \alg{L}$ be the function defined by $u_k (X)=(1+\card{X}\mod k)\times X$ for any $X\neq \omega$ and $u_k(\omega)=\omega$. Then $u_k$ is not smooth. Indeed, if $X$ is a proper infinite subset of $\omega$ then
$\widetilde{u_k}(X)=\{X, 2\times X, \ldots, k\times X\}$. Moreover,  we have $\widetilde{u_l\circ u_k}=\widetilde{u_l}\circ \widetilde{u_k}$ if and only if $l$ and $k$ are coprime. 
\end{enumerate}
\end{ex}

\section{Concluding remarks and further research}
In this paper, we have considered the question of extending functions between algebras to their profinite completions in the setting of finitely generated quasivarieties. Our answer is only partly satisfactory as we provide  an extension which is a multi-map rather than a function. This multi-extension has strong continuity properties and there are interesting cases in which it turns out to be a function. Moreover, the construction of the multi-extension shed lights (Corollary \ref{cor:lat}) on the existence of two canonical extensions in the bounded distributive lattice setting.

We now identify some topics of further research.

\begin{enumerate}
\item Topology $\delta$ (Definition \ref{defn:topo}) is one of the possible topologies in which $\alg{A}$ can be defined as the algebra of isolated points of $\cate{X}^b(\alg{A}^*, \utilde{M})$ and is duality dependent. A general study of the topologies that enjoy this property would lead to other multi-extensions which could be `closer' to a function than the relation $\widetilde{u}$ considered in this work.

\item  Sufficient conditions for $\widetilde{u}$ to be smooth are needed.

\item Canonical extensions have proved to be a useful tool to look for \textsc{Kripke} complete modal logics. Fields of applications of the techniques developed  in this paper should be found outside the lattice-based setting
%

\item Median algebras and median semilattices are equivalent. Natural extensions of median algebras and canonical extensions of their median semilattices \cite{Gouveia2014} should be compared. This constitutes a topic of current investigation.

\end{enumerate}


\appendix
\section{Profinite completions of Boolean products}
The generalization of  Theorem \ref{thm:prod} to Boolean products depends on the possibility to express emptiness in the dual space in terms of formulas in the algebra, as seen in the next result. Recall the following notation: if $a\in\alg{A}$ and $m\in\alg{M}$ we denote by $[a:m]$ the set $\{\psi \in \cate{A}(\alg{A}, \alg{M})\mid \psi(a)=m\}$. The family $\{[a:m]\mid a\in \alg{A}, m\in\alg{M}\}$ is a basis of clopen subsets of $\alg{A}^*$.

The following theorem generalizes the developments in \cite{Hansoul1984} about Boolean products of bounded distributive lattices.

\begin{thm}\label{thm:bool_prod}
Assume that $\utilde{M}$ yields a logarithmic duality for $\var{A}$ and that $ \alg{M}$ is of finite type. Let $\alg{A}$  be a Boolean product of the family $(\alg{A}_i)_{i\in I}$ of algebras of $\var{A}$. If for every $n\in \N$ and every $m_1, \ldots, m_n\in M$ there is an open formula $\phi(x_1, \ldots, x_n)$ in the language of $\alg{M}$ such that for every $i \in I$ and every $a_1, \ldots, a_n \in \alg{A}_i$, it holds
\[
\alg{A}_i^*=\bigcup_{\lambda=1}^{n}[a_\lambda: m_\lambda]\quad  \iff \quad \alg{A}_i\models\phi(a_1, \ldots, a_n),
\]
then $\alg{A}^\delta$ is $\var{A}_\iota$-isomorphic to $\prod_{i\in I}\alg{A}_i^\delta$. 
\end{thm}

\begin{proof}
Let $f: \alg{A}\hookrightarrow \prod_{i\in I} \alg{A}_i$ be a Boolean representation of the family $(\alg{A}_i)_{i\in I}$ of algebras of $\cate{A}$. For every $i\in I$  we denote by $\rho_i$ the embedding $(\pi_i)^*:A_i^* \hookrightarrow \amalg\{\alg{A}_i^* \mid i \in I\}$ where $\pi_i$ denotes the projection map from $\prod_{i\in I}A_i$ onto its $i$-th factor $\alg{A}_i$, \emph{i.e.}, $\rho_i$ is the map defined by $\rho_i(\psi)=\psi\circ\pi_i$. Let $X$ be the set $\bigcup\{\rho_i(\alg{A}_i^*)\mid i \in I\}$. Since $\utilde{M}$ yields a logarithmic duality for $\var{A}$, it is not difficult to see that $\bigcup\{\rho_i(\alg{A}_i^*) \mid i \in J\}$ is isomorphic to $\amalg\{\alg{A}_i^* \mid i \in J\}$ for every finite subset $J$ of $I$. It follows that $X$ is a (not necessarily closed) substructure of $\amalg\{A_i^*\mid i \in I\}$ (such a verification involves only finitely many terms $\rho_i(\alg{A}_i^*)$). In particular, $X$ can be seen as
\begin{equation}\label{eqn:sumbb}
X=\amalg\{(A_i^*)^\flat \mid i \in I\}.
\end{equation}
We are going to prove that $X$ can be equipped with a Boolean topology $\tau$ to obtain a topological structure that is isomorphic to $A^*$ and that is embeddable into $\amalg\{\alg{A}_i^*\mid i \in I\}$.

We define the topology $\tau$ on $X$ as the topology generated by the sets
\[
[a:m]=\bigcup\{[\pi_i(f(a)):m]\mid i \in I\}, \quad a \in \alg{A}, m \in \alg{M}.
\]
The topology $\tau$ is clearly finer than the topology induced on $X$ by $\amalg\{\alg{A}_i^* \mid i \in I\}$. Let us show that $\struc{X,\tau}$ is Boolean. It suffices to prove that it is compact. Assume that $X=\bigcup\{[a_\lambda:m_\lambda]\mid \lambda \in L\}$ for some $a_\lambda\in \alg{A}$ and $m_\lambda \in \alg{M}$. For every $i \in I$, the family $\{[\pi_i(f(a_\lambda)): m_\lambda]\mid \lambda \in L\}$ is an open covering of $\rho_i(\alg{A}_i^*)$ and there is a finite subset $L_i$ of $L$ such that
\begin{equation}\label{eq:sqdf}
\rho_i(A_i^*)=\bigcup \{[\pi_i(f(a_\lambda)): m_\lambda]\mid \lambda \in L_i\}.
\end{equation}
By hypothesis, for every $i\in I$ there is an open formula formula $\phi_{in_i}$ with $n_i$ variables  (where $n_i$ denotes $\card{L_i}$) such that
identity (\ref{eq:sqdf}) is equivalent to
\begin{equation}\label{eq:sqdf2}
\alg{A}_i \models \phi_{in_i}((\pi_i(a_{\lambda}))_{\lambda  \in L_i}).
\end{equation}
Now, for every $i\in I$ let $\Omega_i$ be  defined by
\[
\Omega_i=\{j \in I \mid \alg{A}_j \models\phi_{in}((\pi_j(f(a_{\lambda})))_{\lambda  \in L_i})\}.
\]
The family $\{\Omega_i\mid i \in I\}$ is an open covering of $I$. By compactness, there is a finite subset $J$ of $I$ such that
\begin{equation}\label{eq:sqdf3}
I=\bigcup\{\Omega_j \mid j \in J\}.
\end{equation}
By combining (\ref{eq:sqdf2}) and (\ref{eq:sqdf3}), we obtain,
\[
X=\bigcup\{\bigcup\{[a_\lambda:m_\lambda]\mid \lambda \in L_j\}\mid j \in J\},
\]
which is a finite open covering of $X$ extracted from $\{[a_\lambda: m_\lambda]\mid \lambda \in L\}$.

Let us denote by $g$ the restriction of $f^*$ to $X$. Hence, for any $\rho_i(\psi)\in \rho_i(\alg{A}_i)$, we have $g(\rho_i(\psi))=\psi\circ\pi_i\circ f$. We aim to prove that $g$ is an $\cate{X}$-isomorphism between $\struc{X,\tau}$ and $\alg{A}^*$.

First we prove that $g$ is a $\cate{X}^\flat$-embedding. We have to prove that if  $r$ represents an $n$-ary relation or the graph of a (partial) operation  in the language of $\utilde{M}$ and if $\psi_1, \ldots, \psi_n \in X$, we have the following equivalence
\begin{equation}\label{eqn:zpygoi}
(\psi_1, \ldots, \psi_n) \in r^{X} \Leftrightarrow (g(\psi_1), \ldots, g(\psi_n))\in r^{\alg{A}^*}.
\end{equation}

Let $J$ be a finite subset of $I$ such that $\{\psi_1, \ldots, \psi_n\}\subseteq \bigcup\{\rho_j(\alg{A}_j^*)\mid j \in J\}$. Let us denote by $Y$ the latter set. We have already noted  that $Y$, considered as a substructure of $\amalg\{A_i^*\mid i \in I\}$  is isomorphic to $\amalg\{A_j^*\mid j \in J\}$. Since $f:A\hookrightarrow \prod_{i \in I}A_i$ is a Boolean representation of $A$, the map $f_J:A\rightarrow \prod_{j\in J}A_j:a \mapsto (\pi_j(a))_{j \in J}$ {is onto}. Hence, the dual map $f_J^*:Y \rightarrow A^*$ is an embedding and is clearly equal to the restriction of $g$ to $Y$. Then, it follows successively
\begin{eqnarray*}
(\psi_1, \ldots, \psi_n)\in r^X & \Leftrightarrow & (\psi_1, \ldots, \psi_n)\in r^Y\\
& \Leftrightarrow & (f_J^*(\psi_1), \ldots, f_J^*(\psi_n))\in r^{\alg{A}^*}\\
& \Leftrightarrow & (g(\psi_1), \ldots, g(\psi_n))\in r^{\alg{A}^*},
\end{eqnarray*}
which establishes equivalence (\ref{eqn:zpygoi}), as required.

Finally, since $g$ is the restriction on $X$ of a continuous map, it is a continuous map for the induced topology on $X$. From the fact that $\tau$ is finer than the induced topology we eventually conclude that $g:\struc{X,\tau}\rightarrow A^*$ is an $\cate{X}$-embedding. We deduce that $\struc{X,\tau}\in \cate{X}$.

For the last part of the proof, we show that the evaluation map
\[
h:\alg{A}\rightarrow \cate{X}(X, \utilde{M}):a \mapsto h(a): \rho_i(\psi)\mapsto \psi(\pi_i(f(a)))
\]
is an isomorphism. It is clearly a homomorphism. Moreover, if $a, b \in \alg{A}$ and $a\neq b$ then there is an $i \in I$ such that $\pi_i(f(a))\neq\pi_i(f(b))$, \emph{i.e.}, such that $e_{\alg{A}_i}(\pi_i(f(a)))\neq e_{\alg{A}_i}(\pi_i(f(b)))$. Let $\psi \in \alg{A}_i^*$ with  $e_{\alg{A}_i}(\pi_i(f(a)))(\psi)\neq e_{\alg{A}_i}(\pi_i(f(b)))(\psi)$. It means that $\psi(\pi_i(f(a)))\neq \psi(\pi_i(f(b)))$ which proves that $h$ is one-to-one. Moreover, since $h^*=g$ and since $g$ is an embedding, we deduce that $h$ is onto and so, is an isomorphism.

Hence, it follows successively that 
\[
\alg{A}^\delta\simeq \cate{X}^\flat(\alg{A}^*, \utilde{M})\simeq \cate{X}^\flat(X, \utilde{M})\simeq \cate{X}^\flat(\amalg_{i\in I}(\alg{A}_i^*)^\flat, \utilde{M}),
\]
where we have used (\ref{eqn:sumbb}) to obtain the latter isomorphism. Then, we obtain
\[
 \cate{X}^\flat(\amalg_{i\in I}(\alg{A}_i^*)^\flat, \utilde{M}) \simeq \prod_{i\in I}\cate{X}^\flat(\alg{A}_i^*, \utilde{M})\simeq \prod_{i\in I}\alg{A}_i^\delta
\]
where the first isomorphism is obtained by partnership duality  \cite[Theorem 2.4]{Davey2012}) and {is also an $\var{A}_\iota$-isomorphism}.
\end{proof}

\end{document}